\documentclass[12pt]{article}
\usepackage{amsmath, amssymb, amsthm, url}

\newtheorem{theorem}{Theorem}[section]
\newtheorem{lem}[theorem]{Lemma}

\def\3{\subset }
\def\4{\subseteq }
\def\ov{\overline}
\def\alege#1#2{\left(\begin{array}{c}{#1}\\ {#2}\end{array}\right)}

\def\0{\leqno}
\def\a{{\alpha}}
\def\barr{\begin{array}}
\def\earr{\end{array}}
\def\dd{\displaystyle}

\def\bld#1#2{{\buildrel{#1}\over{#2}}}
\def\st#1#2{{\mathrel{\mathop{#2}\limits_{#1}}{}\!}}
\def\stb#1#2#3{{\st{{#1}}{\bld{{#2}}{#3}}{}\!}}
\def\xmare#1#2{\stb{#1}{#2}{\mbox{\Huge$\times$}}}

\title{\bf On the total number of principal series of a finite abelian group}
\author{Lucian Bentea \and Marius T\u arn\u auceanu}
\date{October 1, 2010}

\begin{document}

\maketitle

\begin{abstract}
In this note we give a bijective proof for the explicit formula
giving the total number of principal series of the direct product
$\mathbb{Z}_{p^{\alpha_1}} \times \mathbb{Z}_{p^{\alpha_2}}$,
where $p$ is a prime number. This new proof is easier to
generalize to arbitrary finite abelian groups than the original
direct calculation method.
\end{abstract}

\noindent{\bf MSC (2000):} Primary 05E15; Secondary 20D15, 20D30,
20D40, 20D60.

\noindent{\bf Key words:} finite abelian groups, principal series,
recurrence relations, ballot numbers, weighted lattice paths

\section{Introduction}

The starting point for our discussion is given by the paper
\cite{Ste}, where the number of maximal chains of subgroups in a
finite nilpotent group has been studied. We have proved that this
combinatorial problem on subgroup lattices can be reduced to
finite $p$-groups. The above number is explicitly computed in
\cite{Ste} for finite elementary abelian $p$-groups, finite abelian
$p$-groups of type $\mathbb{Z}_{p^{\a_1}}\times\mathbb{Z}_{p^{\a_2}}$ and finite
$p$-groups possessing a maximal subgroup which is cyclic.
The main goal of the present note is to give
an alternative proof of Proposition 3.2.1 of \cite{Ste},
which can be more easily generalized to the case
of \emph{arbitrary} finite abelian groups
than the original direct calculation method.
The technique that will be used is related to weighted words
encoding weighted lattice paths in the plane.

First of all, we recall some basic notions and results of group
theory, according to \cite{Suz}. Let $G$ be a finite group. A chain
of subgroups in $G$ of type
$$\{1\}=H_0\subset H_1\subset H_2\subset\cdots\subset H_q=G \0(*)$$ is
called a {\it maximal chain} if it is not properly included in
another chain. Moreover, if $G$ is nilpotent, then $H_{i-1}$ is a
normal subgroup of $H_i$ for any $i=\ov{1,q}$, therefore the
maximal chain $(*)$ is in fact a {\it composition series} of $G$
(see Definition 5.8 (Chapter 1) of \cite{Suz}, I). We also infer
that all maximal chains of subgroups of a nilpotent group are of
the same length. In the particular case when $G$ is abelian, all
subgroups $H_i$ of the maximal chain $(*)$ are normal in $G$ and
so $(*)$ is a {\it principal series} of $G$ (see Definition 1.11
(Chapter 2) of \cite{Suz}, I).

In the following, let $G$ be a finite nilpotent group of order
$p_1^{n_1}p_2^{n_2}...p_m^{n_m}$ $(p_1,p_2,...,p_m$ distinct
primes) and $L(G)$ be the subgroup lattice of $G$. We shall denote
by $n(G)$ the number of all maximal chains of subgroups in $G$ and
by $\alege{n_1+n_2+...+n_m}{n_1,n_2,...,n_m}$ the multinomial
coefficient $\dd\frac{(n_1+n_2+...+n_m)!}{n_1!n_2!...n_m!}.$ Under
our hypothesis, it is well-known that $G$ can be written as the
direct product of its Sylow subgroups
$$G=\xmare{i=1}{m}G_i,$$
where $|G_i|=p_i^{n_i}$, for all $i=\overline{1,m}.$  Since the
subgroups of a direct product of groups having coprime orders are
also direct products (see Corollary of (4.19), \cite{Suz}, I), one
obtains that
$$L(G)=\xmare{i=1}{m}L(G_i).$$
The above lattice direct decomposition  is frequently used to
reduce many problems on $L(G)$ to the subgroup lattices of finite
$p$-groups (for example, see \cite{Sch} and \cite{Tar}). By using this
technique, in \cite{Ste} we have shown the following theorem.

\bigskip\noindent{\bf Theorem A.} {\it The numbers $n(G)$ and $n(G_i)$, $i=1,2,...,m$, are connected by
the equality
$$n(G)=\alege{n_1+n_2+...+n_m}{n_1,n_2,...,n_m} \prod_{i=1}^m n(G_i).$$}

So, in order to compute the number of maximal chains of subgroups
of a finite nilpotent group (and, in particular, of a finite
abelian group), we must focus only on finite $p$-groups. The
structure of these groups is exhaustively determined in the
abelian case. More precisely, by the fundamental theorem of
finitely generated abelian groups, a finite abelian $p$-group has
a direct decomposition of type $\xmare{i=1}{k}\mathbb{Z}_{p^{\a_i}},$
where $p$ is a prime and $1\le\a_1\le\a_2\le...\le\a_k.$ The
subgroups of such a group have been studied in \cite{Bir}. Clearly,
the number of its maximal chains of subgroups (that is, of its
principal series) is an integer valued function that depends on
$p$ and on $\a_1, \a_2,...,\a_k,$ say $f_p(\a_1,\a_2,...\a_k).$
Remark that, for a fixed $p$, this function is symmetric in $\a_1,
\a_2,...,\a_k.$ In the following our aim is to determine it in an
explicit manner.\bigskip

Obviously, if $k=1$ we have $f_p(\a_1)=1$. For $k=2$, the next
auxiliary result has been used in \cite{Ste} to compute
$f_p(\a_1,\a_2).$

\bigskip\noindent{\bf Lemma B.} {\it The number $f_p(\a_1,\a_2)$ of all principal series of the direct
product $\mathbb{Z}_{p^{\a_1}}\hspace{-1mm}\times\mathbb{Z}_{p^{\a_2}}$ satisfies
the following recurrence relation:
$$f_p(\a_1,\a_2)=pf_p(\a_1-1,\a_2)+f_p(\a_1,\a_2-1),\mbox{ for all
}1\le\a_1\le\a_2.$$}

In view of Lemma B, the quantity $f_p(\a_1,\a_2)$ was obtained (by
a direct calculation) in Proposition 3.2.1 of \cite{Ste}.

\bigskip\noindent{\bf Theorem C.} {\it The following equality holds:
$$f_p(\a_1,\a_2)=1+\dd\sum_{i=1}^{\a_1}\left[\alege{\a_1{+}\a_2}{i}{-}\alege{\a_1{+}\a_2}{i{-}1}\right]p^i, \mbox{ for all
}1\le\a_1\le\a_2.$$}

\section{A bijective proof of Theorem C}

Our new proof of Theorem C uses a bijective argument, that can be
easily extended to arbitrary finite abelian groups. The explicit
formula that we find is also expressed using ballot numbers (see
\cite{And}). We use combinatorics on words encoding weighted
lattice paths in the plane with standard North and East steps. Our
proof is similar to the one given for Proposition 2 in \cite{Sul}
related to a simple card guessing game, also investigated in
\cite{Pro} (example 5) using generating functions and the kernel
method. The difference is that we use no geometrical arguments,
which we believe makes the proof easier to generalize to higher
dimensions. The kernel method also seems difficult to generalize,
starting from the three dimensional case $k = 3$.

Let $\Sigma := \{N, E\}$ denote the alphabet used to encode North and East
steps in the lattice $\mathbb{N} \times \mathbb{N}$. A path in this lattice,
starting from the origin, is therefore encoded as a word $\alpha \in \Sigma^*$,
where $\Sigma^*$ denotes the Kleene closure of the alphabet $\Sigma$.
Given a fixed and arbitrary word $\alpha \in \Sigma^*$, denote by
$\alpha[k, l] := \alpha_k \, \alpha_{k + 1} \, \ldots \, \alpha_l \in \Sigma^*$
the part of the word $\alpha$ starting at position $k$ and ending at position $l$,
where $1 \leq k \leq l \leq |\alpha|$. We use the notations $|\alpha|$ for
the length of the word $\alpha$ and $\alpha_i$ for the $i$-th symbol of the word $\alpha$.
Also denote by $\alpha[k] := \alpha[1, k]$ the prefix of $\alpha$ consisting of
its first $k$ symbols, with $1 \leq k \leq |\alpha|$.
By the \emph{reflection} of a word $\alpha \in \Sigma^*$ we understand the
word denoted by $\alpha' \in \Sigma^*$ obtained from $\alpha$ by interchanging
its $N$ symbols with $E$ symbols and vice-versa.
Consider the functions $\# N, \# E : \Sigma^* \to \mathbb{N}$ defined through
\[
\# N(\alpha) := \# \{\, i \;;\; 1 \leq i \leq |\alpha|,\ \alpha_i = N \,\}
\]
and
\[
\# E(\alpha) := \# \{\, i \;;\; 1 \leq i \leq |\alpha|,\ \alpha_i = E \,\},
\]
for all words $\alpha \in \Sigma^*$, where $\# A$ denotes the number of elements of the set $A$.
It follows that the pair $(\# E(\alpha), \# N(\alpha)) \in \mathbb{N} \times \mathbb{N}$
specifies the coordinates of the point where the lattice path encoded by $\alpha$ ends.
We say that the $i$-th symbol of $\alpha$ is a \emph{diagonal symbol} provided that
$\# N(\alpha[i]) = \# E(\alpha[i])$, which is the same with saying that the lattice path
encoded by the prefix $\alpha[i]$ ends on the diagonal $y = x$.
Given a diagonal symbol $\alpha_i$, we call $\alpha_{i + 1}$ a \emph{jump symbol}
if $\alpha_{i + 1} = \alpha_i$, which means that the symbol $\alpha_{i + 1}$ causes the lattice path
encoded by the prefix $\alpha[i + 1]$ to go from one side of the diagonal $y = x$ to the other.
By convention, we always take the first symbol of a word $\alpha \in \Sigma^*$ to be a jump step.
We say that a word $\beta \in \Sigma^*$ is
\emph{subdiagonal} if it satisfies $\# E(\beta[i]) \geq \# N(\beta[i])$
for all $i \in \{1, 2, \ldots, |\beta|\}$.
This means that the lattice path encoded by $\beta$ always stays below the diagonal $y = x$,
possibly touching it. If the last symbol of $\beta$ is diagonal, then $\beta$
is called a \emph{lower Catalan word}. Similarly, we say that $\gamma \in \Sigma^*$
is \emph{superdiagonal} if
$\# N(\gamma[i]) \geq \# E(\gamma[i])$ for all $i \in \{1, 2, \ldots, |\gamma|\}$.
This is the same with saying that the lattice path encoded by $\gamma$ always stays
above the diagonal $y = x$, possibly touching it. If the last symbol of $\gamma$ is diagonal,
then $\gamma$ is called an \emph{upper Catalan word}.
Consider the function $J : \Sigma^* \to 2^{\mathbb{N}}$ that maps each word $\alpha \in \Sigma^*$
to the set of indices $J(\alpha)$ of all its jump symbols and is defined through:
\[
J(\alpha) := \{\, i \;;\; 1 \leq i \leq |\alpha|, \ \mbox{$\alpha_i$ is a jump symbol} \,\}.
\]
Then the following word decomposition result holds.

\begin{lem}
Let $\alpha \in \Sigma^*$ be a fixed and arbitrary word and denote by
$\{j_1, j_2, \ldots, j_r\} := J(\alpha)$ the set of indices of its jump symbols, where $r \geq 1$.
Consider the words $\alpha^1$, $\alpha^2$, \ldots, $\alpha^r$ defined through
$
\alpha^k := \alpha[j_k, j_{k + 1} - 1]$
for $1 \leq k < r$, and $\alpha^r := \alpha[j_r, |\alpha|]$.
Then $\alpha$ can be written as the concatenation
\begin{equation}
\alpha = \alpha^1 \, \alpha^2 \, \ldots \, \alpha^r,
\end{equation}
where the words $\alpha^1$, $\alpha^2$, \ldots, $\alpha^{r - 1}$ form an alternating sequence
of lower and upper Catalan words, by definition of the jump function $J$,
and $\alpha^r$ is either subdiagonal or superdiagonal. \qed
\end{lem}

A \emph{symbol weight function} is a function
$\omega : \Sigma^* \times \mathbb{N} \setminus \{0\} \to \mathbb{R}$
that assigns a real value to each symbol of a word $\alpha \in \Sigma^*$,
also called its \emph{weight}. Thus, if $(\alpha, j) \in \Sigma^* \times \mathbb{N}$
is a fixed and arbitrary pair, then $\omega$ maps this pair to the weight
$\omega(\alpha, j) \in \mathbb{R}$ of the $j$-th symbol in the word $\alpha$.
By convention, we put $\omega(\alpha, k) = 0$ provided that $k > |\alpha|$,
for all words $\alpha \in \Sigma^*$. Given a symbol weight function $\omega$,
we define the \emph{weight of a word} to be the product of the weights of all its symbols.
In other words, the weight $\omega(\alpha) \in \mathbb{R}$ of a fixed and arbitrary word
$\alpha \in \Sigma^*$ is defined through
$\omega(\alpha) := \prod_{i = 1}^{|\alpha|} \omega(\alpha, i)$.
We also define the weight of a finite set of words $\mathcal{A} \subset \Sigma^*$ to
be the sum of the weights of all its elements,
$\omega(\mathcal{A}) := \sum_{\alpha \in \mathcal{A}} \omega(\alpha)$.
These definitions and Lemma 2.1 allow us to calculate the weight
of a word $\alpha \in \Sigma^*$ by calculating the weight of each word
in its decomposition (1), independently.
To be more precise, we have the following result.

\begin{lem}
Let $\alpha \in \Sigma^*$ be a fixed, arbitrary word and
denote by $C \subseteq \{1, 2, \ldots, r - 1\}$ the set of indices
of the upper Catalan words in the decomposition (9) of $\alpha$.
Then the weight of $\alpha$ can be calculated through:
\begin{equation}
\omega(\alpha) =
\left[\prod_{i \in C} \omega(\alpha_i)\right]
\left[\prod_{j \in \{1, 2, \ldots, r - 1\} \setminus C} \omega(\alpha_j)\right]
\omega\left(\alpha^r\right). \qed
\end{equation}
\end{lem}

Let $\lambda \in \mathbb{R} \setminus \{0\}$ be a fixed and arbitrary
non-zero real number and consider the symbol weight function
$\omega : \Sigma^* \times \mathbb{N} \setminus \{0\} \to \mathbb{R}$
defined through:
\begin{equation}
\omega(\alpha, j) :=
\begin{cases}
\lambda, & \alpha_j = E, \ \ \# N(\alpha[j]) > \# E(\alpha[j]) \\
& \qquad\qquad\qquad \mbox{or} \\
& \alpha_j = N, \ \ \# N(\alpha[j]) \leq \# E(\alpha[j]), \\
\\
1, & \mbox{otherwise}.
\end{cases}
\end{equation}
Denote by $\mathcal{P}(x, y) := \{\, \alpha \in \Sigma^* \;;\; \# E(\alpha) = x, \ \# N(\alpha) = y \,\}$
the set of words in $\Sigma^*$ encoding all plane lattice paths that start from the origin
and end at $(x, y) \in \mathbb{N} \times \mathbb{N}$.
Consider the double indexed sequence $\left(f(x, y)\right)_{x, y \geq 0}$
whose elements are defined through $f(x, y) := \omega\left(\mathcal{P}(x, y)\right)$,
for each $x, y \geq 0$, where $\omega$ is the weight function
determined by the symbol weight function in (3).
Then the elements $f(x, y) \in \mathbb{R}$ satisfy the recurrence relations
\begin{equation}
f(x, y) = \lambda f(x - 1, y) + f(x, y - 1),
\quad \forall y > x
\end{equation}
and
\begin{equation}
f(x, y) = f(x - 1, y) + \lambda f(x, y - 1),
\quad \forall x \geq y,
\end{equation}
where $x, y \geq 1$. They also obey the initial conditions:
\begin{equation}
f(0, y) = f(x, 0) = 1,
\quad \forall x, y \geq 0.
\end{equation}
The problem is to solve the partial difference equation given by
(4), (5) and (6). In other words, we ask for an explicit formula for
$f(x, y)$, given fixed and arbitrary indices $x, y \geq 0$.
Setting $\lambda = p$, we notice that the function $f(x, y)$
satisfies the same recurrence relation
as the function $f_p(\alpha_1, \alpha_2)$ in Lemma B,
counting the number of all principal series of the direct product
$\mathbb{Z}_{p^{\alpha_1}} \times \mathbb{Z}_{p^{\alpha_2}}$.
Thus, finding an explicit formula for $f(x, y)$ is the same
with proving Theorem C, which we do in what follows.
Denote by $\mathcal{P}_n(x, y) := \{\, \alpha \in \mathcal{P}(x, y) \;;\; \omega(\alpha) = \lambda^n \,\}$
the set of all words of weight $\lambda^n$, whose associated lattice paths start from the origin
and end at $(x, y) \in \mathbb{N} \times \mathbb{N}$.
Clearly, the set $\left\{\, \mathcal{P}_n(x, y) \;;\; n \geq 0 \, \right\}$
forms a partition of $\mathcal{P}(x, y)$, which allows us to write
$f(x, y) = \omega\left(\mathcal{P}(x, y)\right) = \sum_{n \geq 0} \omega\left(\mathcal{P}_n(x, y)\right)$.
The weight of $\mathcal{P}_n(x, y)$ is also given by
$\omega\left(\mathcal{P}_n(x, y)\right) = \sum_{\alpha \in \mathcal{P}_n(x, y)} \omega(\alpha)$,
for fixed and arbitrary $n \geq 0$.
Since the weight of all words in $\mathcal{P}_n(x, y)$ is equal to $\lambda^n$, we have
$\omega\left(\mathcal{P}_n(x, y)\right) = \# \mathcal{P}_n(x, y) \ \lambda^n$,
which gives a more explicit formula for $f(x, y)$:
\begin{equation}
f(x, y) = \sum_{n \geq 0} \# \mathcal{P}_n(x, y) \ \lambda^n.
\end{equation}
We now aim at finding an explicit formula for the number of elements of $\mathcal{P}_n(x, y)$.
To achieve this, we first need to obtain a general formula
for the weight of an arbitrary word in $\mathcal{P}(x, y)$.
Let $x, y \geq 0$ be such that $x < y$ and let $\alpha \in \mathcal{P}(x, y)$
be a fixed and arbitrary word. Consider the decomposition (1)
of $\alpha$ and let $k \in \{1, 2, \ldots, r - 1\}$ be fixed.
By definition (3) of our particular symbol weight function $\omega$,
if $\alpha^k$ is a lower Catalan word, then its weight is given by
$\omega\left(\alpha^k\right) = \lambda^{\# N\left(\alpha^k\right)} = \lambda^{\# E\left(\alpha^k\right)}$.
On the other hand, if $\alpha^k$ is an upper Catalan word, then its weight can be calculated using
$\omega\left(\alpha^k\right) = \lambda^{\# E\left(\alpha^k\right) - 1}$.
Since $x < y$, the word $\alpha^r$ is superdiagonal and its weight is given by
$\omega\left(\alpha^r\right) = \lambda^{\# E\left(\alpha^r\right)}$.
Denote by $C \subseteq \{1, 2, \ldots, r - 1\}$ the set of indices
of the upper Catalan words in the decomposition (1) of $\alpha$.
On account of Lemma 2.2 and our previous calculations, the following relation is obtained
\begin{equation}
\log_{\lambda} \omega(\alpha) =
-\# C + \# E\left(\alpha^r\right) +
\sum_{i \in C} \# E\left(\alpha^i\right) +
\sum_{j \in \{1, 2, \ldots, r - 1\} \setminus C} \# E\left(\alpha^j\right),
\end{equation}
which can also be written as $\log_{\lambda} \omega(\alpha) = \# E(\alpha) - \# C$.
From this relation, the following result is obtained.

\begin{lem}
Let $n \geq 0$ and $x, y \geq 0$ be fixed, with $x < y$.
Then the weight of a word $\alpha \in \mathcal{P}(x, y)$ is $\lambda^n$ if and only if
the number of upper Catalan words in its decomposition is equal to $x - n$. \qed
\end{lem}

Following a similar argument, it can also be shown that the next result holds
for the case when $x \geq y$.

\begin{lem}
Let $n \geq 0$ and $x, y \geq 0$ be fixed, with $x \geq y$.
Then the weight of a word $\alpha \in \mathcal{P}(x, y)$ is $\lambda^n$ if and only if
the number of upper Catalan words in its decomposition is equal to $y - n$. \qed
\end{lem}

Denote by
$\mathcal{C}_k(x, y) := \{\, \alpha \in \mathcal{P}(x, y) \;;\; \mbox{$\alpha$ contains $k$ upper Catalan words} \,\}$
the set of all words containing $k \geq 0$ upper Catalan words
in their decomposition. For fixed and arbitrary $n \geq 0$, Lemmas 2.3 and 2.4 can be restated as
\begin{equation}
\mathcal{P}_n(x, y) = \mathcal{C}_{x - n}(x, y), \quad 0 \leq x < y
\end{equation}
and
\begin{equation}
\mathcal{P}_n(x, y) = \mathcal{C}_{y - n}(x, y), \quad 0 \leq y \leq x,
\end{equation}
correspondingly. Our aim is now to count the number of elements
in the set $\mathcal{C}_k(x, y)$, for fixed and arbitrary $k \geq 0$.
In order to achieve this, we need another definition.
A word $\alpha \in \Sigma^*$ is said to be a \emph{ballot word} provided that
$\# E(\alpha[i]) \leq \# N(\alpha[i])$ for all $i \in \{1, 2, \ldots, |\alpha|\}$.
Thus, the plane lattice path encoded by a ballot word always stays above the diagonal $y = x$,
possibly touching it. Denote by $\mathcal{B}(x, y)$ the set of all ballot words whose associated
lattice paths end at $(x, y)$, for $x, y \geq 0$ with $x \leq y$.
Then we have the following result.

\begin{lem}
Let $k \geq 0$ and $x, y \geq 0$ be fixed such that $x < y$.
Then the sets $\mathcal{C}_k(x, y)$ and $\mathcal{B}(x - k, y + k)$
have the same number of elements.
\end{lem}

\begin{proof}
For fixed, arbitrary $k \geq 0$ and $x, y \geq 0$ such that $x < y$,
consider the function $\varphi : \mathcal{C}_k \to \mathcal{B}(x - k, y + k)$
defined as follows. If $\alpha \in \mathcal{C}_k$ is a fixed and arbitrary word,
denote by $C \subseteq \{1, 2, \ldots, r - 1\}$ the set of indices
of the upper Catalan words in the decomposition (1) of $\alpha$.
For all $i \in C$, $\varphi$ reflects the last symbol of
the upper Catalan word $\alpha^i$, thus replacing it with a $N$ symbol.
For each $j \in \{1, 2, \ldots, r - 1\} \setminus C$,
the function $\varphi$ maps the lower Catalan word
$\alpha^j$ into its reflection $\left(\alpha^j\right)'$.
Also, the function $\varphi$ leaves the last word $\alpha^r$ unchanged.
To conclude, the resulting lattice path associated with the word $\varphi(\alpha)$
always stays strictly above the diagonal $y = x$. Its end point is also offset by
$(-k, +k)$ from the original position, due to the reflection of the last symbol
of each upper Catalan word into a $N$ symbol.
Thus $\varphi(\alpha)$ is a ballot word and we also have $\varphi(\alpha) \in \mathcal{B}(x - k, x + k)$.
To show that $\varphi$ is bijective, we explicitly define its inverse.
Let $\beta \in \mathcal{B}(x - k, y + k)$
be fixed, arbitrary and denote by $\delta_s(\beta) \subset \{1, 2, \ldots, |\beta|\}$
the set defined through
\begin{multline*}
\delta_s(\beta) := \min\, \{\, i \;;\; 1 \leq i \leq |\beta|, \\
\# N(\beta[j]) > \# E(\beta[j]) + s, \ \forall j \in \{i, i + 1, \ldots, |\beta|\} \,\},
\end{multline*}
for $s \geq 0$. In terms of lattice paths, $\delta_s$ specifies the index
of the first step from which the path associated with $\beta$ remains
strictly above the diagonal $y = x + s$.
Consider the function $\theta : \mathcal{B}(x - k, y + k) \to \mathcal{C}_k(x, y)$
that maps the ballot word $\beta$ into the word $\theta(\beta) \in \mathcal{C}_k(x, y)$
containing $k$ upper Catalan words in its decomposition, defined as follows.
Let $m \geq 0$ be the largest integer such that the minimum $\delta_m(\beta)$ exists.
For all $t \geq 0$ such that $2t < m$, $\theta$ reflects the last symbol of each part
of the word $\beta$ starting at index $\delta_{2t}(\beta)$ and ending at index $\delta_{2t + 1}(\beta)$.
Also, $\theta$ reflects all the remaining symbols of $\beta$,
apart from those whose index is at least $\delta_m(\beta)$.
Thus, for all $t \geq 0$ with $2t < m$,
the function $\theta$ maps each part of the word $\beta$,
starting at index $\delta_{2t}(\beta)$ and ending at index $\delta_{2t + 1}(\beta)$,
into an upper Catalan word. The part of $\beta$ starting at index $\delta_m(\beta)$ is left unchanged,
thus being mapped into a superdiagonal word. The rest of the word $\beta$ is mapped
into a series of lower Catalan words, via reflection.
Since we have $\beta \in \mathcal{B}(x - k, y + k)$, it follows that the lattice path
associated with $\beta$ successively leaves the diagonals $y = x$, $y = x + 1$,
\ldots, $y = x + 2k$, proving that $m = 2k$.
Thus $\theta(\beta)$ contains exactly $k$ upper Catalan words in its decomposition,
showing that $\theta$ is well defined. The function $\theta$ is also the inverse of $\varphi$,
making $\varphi$ bijective. It follows that the sets $\mathcal{C}_k(x, y)$ and
$\mathcal{B}(x - k, y + k)$ indeed have the same number of elements.
\end{proof}

On account of relation (9) and Lemma 2.5 it follows that
\[
\# \mathcal{P}_n(x, y) = \# \mathcal{C}_{x - n}(x, y) = \# \mathcal{B}(n, x + y - n),
\]
for $0 \leq x < y$.
Fortunately, an explicit formula exists (\cite{And}) for the \emph{ballot number}
$b(x, y) := \# \mathcal{B}(x, y)$ and it is given by
\[
b(x, y) = \binom{x + y}{x} - \binom{x + y}{x - 1} = \frac{y - x + 1}{y + 1} \binom{x + y}{x},
\]
for all $0 \leq x \leq y$. Thus, the number of words in $\mathcal{P}(x, y)$
whose weight is equal to $\lambda^n$ is given by
$\# \mathcal{P}_n(x, y) = b(n, x + y - n) = \binom{x + y}{n} - \binom{x + y}{n - 1}$,
which leads to an explicit formula for $f(x, y)$ via relation (7)
\begin{equation}
f(x, y) =
\sum_{n \geq 0} b(n, x + y - n) \lambda^n =
\sum_{n \geq 0} \left[ \binom{x + y}{n} - \binom{x + y}{n - 1} \right] \lambda^n,
\end{equation}
provided that $x, y \geq 0$ satisfy $x < y$. For the case when $x \geq y$, the function
$\varphi$ in the proof of Lemma 2.5 is defined to also reflect the last word $\alpha^r$
in the decomposition of $\alpha \in \mathcal{C}_k(x, y)$, rather than leaving it unchanged.
Following a similar argument we conclude that the sets $\mathcal{C}_k(x, y)$
and $\mathcal{B}(y - k, x + k)$ have the same number of elements.
Thus, using relation (10) we obtain
\[
\# \mathcal{P}_n(x, y) = \# \mathcal{C}_{y - n}(x, y) = \# \mathcal{B}(n, x + y - n).
\]
provided that $0 \leq y \leq x$. The explicit formula (11) for $f(x, y)$ is therefore valid
for all $x, y \geq 0$. \qed

\section{The general case $k \geq 3$}

The following result extends Lemma B for the general case when $k\geq 3$.

\bigskip\noindent{\bf Lemma D.} {\it The number $f_p(\a_1,\a_2,...,\a_k)$ of all principal series of the direct
product $\xmare{i=1}{k}\mathbb{Z}_{p^{\a_i}}$ satisfies the following
recurrence relation:
$$f_p(\a_1,\!\a_2,...,\!\a_k){=}\dd\sum_{i=1}^{k}p^{k-i}f_p(\a_1,...,\a_i{-}1,...,\a_k),\mbox{ for all
}1\!\le\!\a_1\!\le\!\a_2\!\le\!...\!\le\!\a_k.$$}

By the above lemma we easily infer that $f_p(\a_1,\a_2,...,\a_k)$
is a polynomial in $p$ of degree
$(k-1)\a_1+(k-2)\a_2+...+\a_{k-1}$. Moreover, we remark that the
sum $f(\a_1,\a_2,...,\a_k)$ of all coefficients of
$f_p(\a_1,\a_2,...,\a_k)$ satisfies the recurrence relation
$$f(\a_1,\a_2,...,\a_k)=\dd\sum_{i=1}^{k}f(\a_1,...,\a_i{-}1,...,\a_k)$$
and therefore we have:
$$f(\a_1,\a_2,...,\a_k)=\alege{\a_1+\a_2+...+\a_k}{\a_1,\a_2,...,\a_k}.$$

\bigskip\noindent{\bf Proof of Lemma D.} It is well-known that the group $\xmare{i=1}{k}\mathbb{Z}_{p^{\a_i}}$
possesses $r=\dd\sum_{i=1}^k p^{i-1}$ maximal subgroups, say $M_1,
M_2,...,M_r$. In order to establish the desired recurrence
relation we shall use the following simple remark: every principal
series of $\xmare{i=1}{k}\mathbb{Z}_{p^{\a_i}}$ contains a unique maximal
subgroup. So, the numbers $n(\xmare{i=1}{k}\mathbb{Z}_{p^{\a_i}})$ and
$n(M_s), s=1,2,...,r,$ are connected by the equality
\begin{equation}
n(\xmare{i=1}{k}\mathbb{Z}_{p^{\a_i}})=\dd\sum_{s=1}^rn(M_s).
\end{equation}

This shows that we need to know the structure of maximal subgroups
of $\xmare{i=1}{k}\mathbb{Z}_{p^{\a_i}}$. According to (4.19) of \cite{Suz},
I, such a subgroup $M$ is uniquely  determined by two subgroups
$H_1\4H'_1$ of $\xmare{i=1}{k-1}\mathbb{Z}_{p^{\a_i}}$, two subgroups
$H_2\4H'_2$ of $\mathbb{Z}_{p^{\a_k}}$ and a group isomorphism
$\varphi:H'_1/H_1 \to H'_2/H_2$ (more exactly, $M=\{(a_1,a_2)\in
H'_1\times H'_2\mid\varphi(a_1H_1)=a_2H_2\}).$ Moreover, we have
$$|H'_1|\,|H_2|=|H'_2|\,|H_1|=|M|=p^{\hspace{1mm}\dd\sum_{i=1}^k \a_i-1}$$
and thus we distinguish the following three cases:

{\bf Case 1.}
$|H'_1|=|H_1|=p^{\hspace{1mm}\dd\sum_{i=1}^{k-1}\a_i},\hspace{1mm}
|H'_2|=|H_2|=p^{\a_k-1}.$\medskip

\noindent Then between $H'_1/H_1$ and $H'_2/H_2$ there exists only
the trivial isomorphism $\varphi$, therefore we get:
\begin{equation}
M\cong(\xmare{i=1}{k-1}\mathbb{Z}_{p^{\a_i}})\times\mathbb{Z}_{p^{\a_k-1}}.
\end{equation}

{\bf Case 2.}
$|H'_1|=|H_1|=p^{\hspace{1mm}\dd\sum_{i=1}^{k-1}\a_i-1},
\hspace{1mm} |H'_2|=|H_2|=p^{\a_k}.$\medskip

\noindent Similarly to Case 1, $\varphi$ is the trivial
isomorphism. The subgroup $H'_1=H_1$ of
$\xmare{i=1}{k-1}\mathbb{Z}_{p^{\a_i}}$ can be chosen in
$\dd\sum_{i=1}^{k-1} p^{i-1}$ distinct ways. Then, for every
$j=1,2,...,k-1,$ one obtains
\begin{equation}
M\cong(\xmare{i=1}{j-1}\mathbb{Z}_{p^{\a_i}})\times\mathbb{Z}_{p^{\a_j-1}}\times(\xmare{i=j+1}{k}\mathbb{Z}_{p^{\a_i}})\hspace{1mm} \mathrm{in}\hspace{1mm} p^{k-1-j}\hspace{1mm} \mathrm{cases}.
\end{equation}

{\bf Case 3.}
$|H'_1|=p^{\hspace{1mm}\dd\sum_{i=1}^{k}\a_i},\hspace{1mm}
|H_1|=p^{\hspace{1mm}\dd\sum_{i=1}^{k}\a_i-1},\hspace{1mm}
|H'_2|=p^{\a_k},\ |H_2|=p^{\a_k-1}.$\medskip

\noindent In this case there are $p-1$ distinct isomorphisms
$\varphi$ from $H'_1/H_1$ to $H'_2/H_2$ and $H_1$ can be chosen
again in $\dd\sum_{i=1}^{k-1} p^{i-1}$ distinct ways. Then, for
every $j=1,2,...,k-1,$ one obtains
\begin{equation}
M\cong(\xmare{i=1}{j-1}\mathbb{Z}_{p^{\a_i}})\times\mathbb{Z}_{p^{\a_j-1}}\times(\xmare{i=j+1}{k}\mathbb{Z}_{p^{\a_i}})\hspace{1mm} \mathrm{in}\hspace{1mm} p^{k-1-j}(p-1)\hspace{1mm} \mathrm{cases}.
\end{equation}

Now, by relations (13), (14) and (15), we infer that
$\xmare{i=1}{k}\mathbb{Z}_{p^{\a_i}}$ has one maximal subgroup of type
$(\a_1,...,\a_{k-1},\a_k-1)$ and
$p^{k-1-j}+p^{k-1-j}(p-1)=p^{k-j}$ maximal subgroups of type
$(\a_1,...,\a_j-1,...,\a_k)$, for all $j=\ov{1,k-1}$. Hence (12)
becomes
$$f_p(\a_1,\a_2,...,\a_k)=\dd\sum_{j=1}^{k}p^{k-j}f_p(\a_1,...,\a_j-1,...,\a_k),$$ as desired. \qed

\bigskip
In our alternative proof of Theorem C, we proved that the explicit formula
for $f_p(\alpha_1, \alpha_2)$, counting the total number of principal series of the direct product
$\mathbb{Z}_{p^{\alpha_1}} \times \mathbb{Z}_{p^{\alpha_2}}$, is given by
\[
f_p(\alpha_1, \alpha_2) = \sum_{n \geq 0} b(n, \alpha_1 + \alpha_2 - n) p^n,
\]
where $b(x, y)$ is a ballot number.
Our first open problem is to find an explicit formula for
$f_p(\a_1,\a_2,...,\a_k)$, which we believe is similar to (11) and uses
multidimensional ballot numbers. For a rigurous introduction to multidimensional
ballot numbers we refer to \cite{Nat}.

A class of groups whose structure is strongly
connected to that of finite abelian groups consists of
finite hamiltonian groups, that is finite nonabelian groups having
all subgroups normal. Such a group $H$ is the direct product of a
quaternion group of order 8, a finite elementary abelian 2-group
and a finite abelian group of odd order. Find an explicit formula
for $n(H)$, by using Theorem A and Proposition 3.1.2
of \cite{Ste}.

\vspace*{5ex} \small

\begin{minipage}[t]{6cm}
Lucian Bentea \\
Faculty of Informatics \\
``Al.I. Cuza'' University \\
Ia\c si, Romania \\
e-mail: {\tt lucian\_bentea@yahoo.com}
\end{minipage}
\hfill
\begin{minipage}[t]{6cm}
Marius T\u arn\u auceanu \\
Faculty of  Mathematics \\
``Al.I. Cuza'' University \\
Ia\c si, Romania \\
e-mail: {\tt tarnauc@uaic.ro}
\end{minipage}

\end{document}